\documentclass[11pt,reqno]{amsart}
\usepackage[T1]{fontenc}
\usepackage{lmodern}
\usepackage{amsmath,amssymb,amsthm,mathtools}
\usepackage[margin=1in]{geometry}
\usepackage[expansion=false]{microtype}
\usepackage{enumitem}
\usepackage{needspace}
\usepackage{tikz}
\usetikzlibrary{arrows,positioning,calc}
\tikzset{property/.style={draw,rounded corners=2pt,align=center,
  inner sep=5pt,font=\small},imp/.style={->,>=latex,semithick},
  equiv/.style={<->,>=latex,semithick}}
\newcommand{\step}[1]{\par\smallskip\noindent\emph{#1}\enspace}
\usepackage[colorlinks=true,linkcolor=black,citecolor=black,urlcolor=black]{hyperref}
\hypersetup{pdfauthor={Hwankoo Kim},pdftitle={An overring characterization of Gorenstein Prufer domains}}
\numberwithin{equation}{section}
\newtheorem{theorem}{Theorem}[section]
\newtheorem{proposition}[theorem]{Proposition}
\newtheorem{corollary}[theorem]{Corollary}
\newtheorem{lemma}[theorem]{Lemma}
\theoremstyle{definition}
\newtheorem{definition}[theorem]{Definition}
\newtheorem{example}[theorem]{Example}
\theoremstyle{remark}

\DeclareMathOperator{\Gfd}{Gfd}

\DeclareMathOperator{\fd}{fd}
\DeclareMathOperator{\Tor}{Tor}
\DeclareMathOperator{\Ext}{Ext}
\DeclareMathOperator{\Hom}{Hom}
\DeclareMathOperator{\Frac}{Frac}
\DeclareMathOperator{\Spec}{Spec}

\newcommand{\GF}{\mathcal{GF}}
\newcommand{\wGgldim}{\operatorname{wGgldim}}
\newcommand{\OG}{\mathrm{(O)}}

\setlist[enumerate,1]{label=\textup{(\roman*)},leftmargin=2.4em}
\allowdisplaybreaks[2]
\newtheorem*{maintheorem}{Theorem A}
\title[An overring characterization]{An overring characterization of\protect\\
Gorenstein Pr\"ufer domains}
\author[Kim]{Hwankoo Kim}
\address{Department of Computer Engineering, Hoseo University,
Asan 31499, Republic of Korea}
\email{hkkim@hoseo.edu}
\date{\today}
\begin{document}
\raggedbottom
\begin{abstract}
We prove that a commutative integral domain is Gorenstein Pr\"ufer
if and only if every overring is Gorenstein flat over the original
domain. It is enough to test overrings that are finitely generated as
algebras.
The proof combines finite duality with a conductor-quotient argument.
Under the overring hypothesis, the conductor quotient of a proper
overring generated by two elements as a module is an IF-ring, and degree-one duality lifts first
syzygies of finitely presented quotient modules to totally reflexive
modules. This gives a bound of one for the Gorenstein flat dimension
of every module annihilated elementwise by conductor powers. Quadratic
integral elements provide such conductors at nonnormal localizations,
while normal localizations give principal-ideal neighborhoods.
Constructible compactness yields a finite cover, and a bounded complex
of localizations combines these bounds over the original domain.
Consequently every finitely generated ideal is Gorenstein flat and the
domain is coherent. Examples distinguish finite-type from module-finite
tests and Gorenstein flatness from ordinary flatness.
\end{abstract}
\maketitle

\section{Introduction}

Overrings encode the divisibility of an integral domain without enlarging
its quotient field. The classical criterion (\cite[Theorem 6.10]{LM}) that a domain is Pr\"ufer
if and only if all its overrings are flat turns this ring-theoretic
information into a statement about every finitely generated ideal.
Replacing flat modules by Gorenstein flat modules suggests a corresponding
criterion for Gorenstein Pr\"ufer domains. The replacement is not formal:
Gorenstein flat modules need not be flat, and coherence cannot be
assumed when it is part of the desired conclusion.

The terminology comes from the Gorenstein homological theory developed
from totally reflexive modules and complete resolutions; see
\cite{CFH,Holm}. Qiao and Wang \cite[Theorem~4.2]{QW} characterized
Gorenstein Pr\"ufer domains by coherence and divisoriality of finitely
generated ideals. Hu, Lim, and Zhou \cite[Corollary~2.4]{HLZ} proved
that a domain is Gorenstein Pr\"ufer precisely when every torsion-free
module is Gorenstein flat, equivalently when every injective module
has flat dimension at most one. Thus every overring of a Gorenstein
Pr\"ufer domain is Gorenstein flat over the base domain. The converse
is asked in \cite[Section~9, Problem~13]{GKL}.

We prove the converse for arbitrary commutative integral domains.

\begin{maintheorem}
Let $R$ be a commutative integral domain and let $K=\Frac R$.
The following conditions are equivalent: $R$ is Gorenstein Pr\"ufer;
every overring $R\subseteq T\subseteq K$ is Gorenstein flat as an
$R$-module; and every finitely generated $R$-subalgebra of $K$ is
Gorenstein flat as an $R$-module.
\end{maintheorem}

The forward implication and the module-theoretic characterizations are
existing results of \cite{HLZ}. The contribution here is the unrestricted
converse and the mechanism that proves it. Finite duality first controls
conductors without coherence. If $T=R+Rx$ is a proper overring, the
Gorenstein flatness of all intermediate rings forces $R/(R:T)$ to be an
IF-ring. We then identify the biduality map of the first $R$-syzygy of a
finitely presented module over this quotient. This produces a complete
resolution and extends the resulting dimension bound to arbitrary
conductor-power torsion modules.

The second contribution is a finite patching argument over a fixed
base ring. At each nonnormal localization we construct a quadratic
integral element and its finitely generated conductor. At normal
localizations, a two-chart tensor calculation makes finitely generated
ideals locally principal. The closed conductor sets and the resulting
basic-open sets form an open cover in the constructible topology.
Compactness selects finitely many sets, and a bounded localization
complex combines the homological bounds. Its cohomology is controlled
by a product of conductors. This avoids an unproved local-to-global
criterion for Gorenstein flat modules.

%Figure~\ref{fig:overring-hierarchy} places the new equivalence beside the
%classical one. All module properties in the figure are taken over $R$.
%Neither one-way implication from the classical row to the Gorenstein
%row has an unrestricted converse; the cusp in Example~\ref{ex:overring-cusp} separates them.
%
%\begin{figure}[htbp]
%\centering
%\begin{tikzpicture}[x=1cm,y=1cm]
%\node[property,text width=4.4cm] (p) at (0,0) {$R$ is Pr\"ufer};
%\node[property,text width=4.4cm] (f) at (6.8,0)
% {Every overring is flat};
%\node[property,text width=4.4cm] (g) at (0,-1.7)
% {$R$ is $G$-Pr\"ufer};
%\node[property,text width=4.4cm] (o) at (6.8,-1.7)
% {Every overring is\par Gorenstein flat};
%\node[property,text width=4.4cm] (c) at (0,-3.5)
% {$R$ is coherent and\par every finitely generated ideal\par is Gorenstein flat};
%\node[property,text width=4.4cm] (ft) at (6.8,-3.5)
% {Every finite-type overring\par is Gorenstein flat};
%\draw[equiv] (p)--node[above,font=\footnotesize]{classical}(f);
%\draw[imp] (p)--(g);
%\draw[imp] (f)--(o);
%\draw[equiv] (g)--node[above,font=\footnotesize]{Theorem~\ref{thm:cp-main}}(o);
%\draw[equiv] (g)--(c);
%\draw[equiv] (o)--(ft);
%\end{tikzpicture}
%\caption{Overring tests and their module-theoretic consequences. For an
%integrally closed domain, all six conditions are equivalent.}
%\label{fig:overring-hierarchy}
%\end{figure}

Finite-type and module-finite extensions must be distinguished.
A finite-type overring is finitely generated as an $R$-algebra;
a module-finite overring is finitely generated as an $R$-module.
The latter tests alone do not suffice, even for a two-dimensional
regular local domain; see Example~\ref{ex:modulefinite-insufficient}.
The two-generated module overrings in the proof are intermediate
objects, not a replacement for the quantifier in Theorem~A.

Section~\ref{sec:qcs-preliminaries} establishes finite duality and produces
quadratic elements at nonnormal localizations.
Section~\ref{sec:qcs-return} treats conductor quotients, their biduality,
and modules annihilated by conductor powers.
Section~\ref{sec:cp-patching} proves finite patching, completes the
overring characterization, and gives the separating examples.

\section{Finite duality and quadratic overrings}
\label{sec:qcs-preliminaries}

This section provides the algebraic inputs for the conductor argument.
We first fix the homological conventions and the finite-generation
properties that are available without coherence. We then use finite
duality to distinguish normal localizations from localizations with
proper quadratic module overrings.

\subsection{Conventions and homological preliminaries}

All rings are commutative with identity, all ring homomorphisms preserve
the identity, and all modules are unital. A domain has $1\ne0$.
Auxiliary quotient rings may have zero divisors. An overring of a domain
$R$ is a ring $T$ with $R\subseteq T\subseteq K=\Frac R$.
We use $M^*=\Hom_R(M,R)$ for the ring dual and
$M^+=\Hom_{\mathbb Z}(M,\mathbb Q/\mathbb Z)$ for the character module.
A fractional ideal is a nonzero $R$-submodule $I\subseteq K$ with
$dI\subseteq R$ for some $0\ne d\in R$; put
$I^{-1}=(R:I)$ and $I_v=(R:(R:I))$. It is \emph{divisorial} when
$I=I_v$. For an overring $T$, the \emph{conductor} is
$(R:T)=\{r\in R:rT\subseteq R\}$; it is an ideal of both rings.

\begin{definition}\label{def:overring-homological}
A module $M$ is \emph{Gorenstein flat} if it is a cycle of an exact
complex of flat modules that remains exact after tensoring with every
injective module. The class of such modules is denoted by $\GF(R)$.
Its associated dimension $\Gfd_R M$ is the least length of a finite
resolution by Gorenstein flat modules, and is infinity when no such
resolution exists. We put $\Gfd_R0=0$.
An $R$-module $M$ is of type $\mathrm{FP}_\infty$, or is
\emph{super finitely presented}, if it admits an exact sequence
\[
 \cdots\longrightarrow F_2\longrightarrow F_1
 \longrightarrow F_0\longrightarrow M\longrightarrow0,
\]
where each $F_i$ is a finitely generated free $R$-module.
Equivalently, one may require each $F_i$ to be finitely generated
projective. We denote the class of these modules by
$\mathrm{FP}_\infty(R)$. A module is \emph{totally reflexive}
if it is a cycle of an exact complex of finitely generated free modules
whose $R$-dual is exact. A domain is \emph{Gorenstein Pr\"ufer} if it is
coherent and $\wGgldim(R):=\sup_M\Gfd_RM\le1$.
\end{definition}

Coherence means that every finitely generated ideal is finitely
presented. The definitions and basic properties of complete resolutions
are recalled in \cite{CFH,Holm}. For a complex of finitely generated
free modules, exactness of the ring dual implies exactness after
$\Hom_R(-,Q)$ for every projective $Q$, by direct sums and direct
summands. It also implies injective tensor exactness by finite free
duality. Thus totally reflexive modules are both Gorenstein projective
and Gorenstein flat.

We write $\OG$ for the condition that every overring of $R$ belongs to
$\GF(R)$. The base ring in this condition is always $R$.
The class $\GF(R)$ is resolving and closed under filtered colimits over
arbitrary rings \cite[Corollary~3.12 in the preprint]{SS}; in particular,
no GF-closed hypothesis needs to be imposed. We use the resulting
cokernel, syzygy, and dimension criteria of
\cite[Theorems~2.3, 2.8, and~2.11]{BennisGF}.

A Gorenstein flat module over a domain is torsion-free, since it embeds
in a flat module. Moreover, $\OG$ localizes. If $R_S$ is a localization,
an injective $R_S$-module restricts to an injective $R$-module, since
$R_S\otimes_R-$ is exact. Localizing a complete flat resolution
therefore preserves both exactness and its injective tensor tests.
Every overring of $R_S$ is also an overring of $R$.

We will use the established criterion that a domain is $G$-Pr\"ufer if
and only if every finitely generated ideal is Gorenstein flat.
Indeed, the forward implication follows from \cite[Corollary~2.4]{HLZ}.
For the converse, the ideal sequences give
$\Tor_2^R(E,R/I)=0$ for every injective $E$ and every finitely generated
ideal $I$. The ideal test for flat dimension gives $\fd_RE\le1$;
\cite[Corollary~2.4]{HLZ} then supplies coherence and the
$G$-Pr\"ufer property. Standard character-duality and purity facts are
used as in \cite{WangKim2024}.

\subsection{Finite factorization and reflexive duals}
Finite factorization and an extension property supply the finite duals
needed for conductors.

\begin{lemma}\label{lem:finite-factorization}
Let $R$ be a commutative integral domain with quotient field $K$.
\begin{enumerate}
\item Every homomorphism from a finitely generated $R$-module to a flat
$R$-module factors through a finitely generated free module.
\item If $N$ is a finitely generated Gorenstein flat $R$-module, then
$\Ext_R^1(N,F)=0$ for every flat $R$-module $F$.
\end{enumerate}
\end{lemma}
\begin{proof}
\step{(i) Factoring through a finite free module.}Take a presentation $R^r\to N$ with kernel $L$, and choose
$\ell_1,\ldots,\ell_s\in L$ spanning $K\otimes_R L$. Express the flat
target as a filtered colimit of finitely generated free modules by
Lazard's theorem \cite[Tag~058G]{Stacks}. The map from $R^r$ factors at
some stage, and a later stage kills all the $\ell_i$. For every
$\ell\in L$ there is $0\ne d\in R$ with
$d\ell\in\sum_iR\ell_i$. Torsion-freeness of that free stage then kills
$\ell$, so the factorization descends to $N$.

\step{(ii) Extending maps into a flat module.}Write $F^+=\Hom_{\mathbb Z}(F,\mathbb Q/\mathbb Z)$.
Since $F^+$ is injective, character duality gives
\[
 \Ext_R^1(N,F^{++})\cong\Tor_1^R(F^+,N)^+=0.
\]
Let $h:L\to F$ be a map on the kernel of a presentation by a finitely generated free module
$R^r\to N$. Extend $h$ to $R^r\to F^{++}$. Choose finitely many
relations spanning $K\otimes_RL$ as in (i). The corresponding finite
system of linear equations, whose constants are their images under
$h$, has a solution in $F^{++}$. Purity of $F\to F^{++}$ gives a
solution in $F$. The resulting map $R^r\to F$ agrees with $h$ on
these relations. Clearing denominators and using torsion-freeness of
$F$ proves agreement on all of $L$. Hence $\Ext_R^1(N,F)=0$.
\end{proof}

The evaluation map becomes an isomorphism for finitely generated
Gorenstein flat modules, even before coherence is available.

\begin{lemma}\label{lem:qcs-finite-dual}
Let $R$ be a domain and $M$ a finitely generated Gorenstein flat $R$-module.
Then the evaluation map $M\to M^{**}$ is an isomorphism,
$M^*\in\mathrm{FP}_\infty(R)$, and
$\Ext_R^i(M^*,R)=0$ for $i>0$.  Also $\Ext_R^1(M,F)=0$ for every flat
$R$-module $F$.  If in addition $M\in\mathrm{FP}_\infty(R)$, then $M$ and
$M^*$ are totally reflexive.
\end{lemma}
\begin{proof}
\step{Step 1: a finite free coresolution.}Choose $0\to M\to F_0\to G\to0$ with $F_0$ flat and $G$ Gorenstein flat.
Factor the first map through a finitely generated free module $P$ by
Lemma~\ref{lem:finite-factorization}(i).
The resulting map $M\to P$ is injective.  For every injective $E$, the map
$E\otimes_R M\to E\otimes_R P$ is injective, since its composite into
$E\otimes_R F_0$ is injective.  Hence the cokernel $M_1=P/M$ satisfies
$\Tor_1^R(E,M_1)=0$.  The cokernel criterion
\cite[Theorem~2.3]{BennisGF} makes $M_1$ Gorenstein flat.  Iterating gives
\[
 0\longrightarrow M_j\longrightarrow P_j\longrightarrow M_{j+1}
 \longrightarrow0\quad(j\ge0),\qquad M_0=M,
\]
with all $P_j$ finitely generated free and all $M_j$ finitely generated
Gorenstein flat.

\step{Step 2: finite duals and reflexivity.}By Lemma~\ref{lem:finite-factorization}(ii), dualizing the displayed
short exact sequences gives
\[
 0\longrightarrow M_{j+1}^*\longrightarrow P_j^*
 \longrightarrow M_j^*\longrightarrow0.
\]
In particular every $M_j^*$ is finitely generated, and splicing gives a
resolution by finitely generated free modules of $M^*$.  Dualizing $P_j^*\to M_j^*$
embeds $M_j^{**}$ into $P_j$.  It contains $M_j$ under the evaluation map.
For a finitely generated torsion-free module, clearing denominators in a
fraction-field dual basis shows that evaluation becomes an isomorphism
over $K$.  Here $M_j^*$ is finitely generated as well, so
$M_j^{**}/M_j$ is torsion.  It embeds in the torsion-free module
$P_j/M_j=M_{j+1}$, and is therefore zero.  The dual of the spliced
resolution of $M^*$ is thus the original exact coresolution.  It gives
$\Ext_R^{>0}(M^*,R)=0$.

\step{Step 3: total reflexivity in type $\mathrm{FP}_\infty$.}If $M\in\mathrm{FP}_\infty(R)$, every syzygy in a resolution by finitely generated free modules is finitely generated and Gorenstein flat.  The degree-one
vanishing already proved gives $\Ext_R^{>0}(M,R)=0$.  Splicing this
resolution with the dual of a resolution by finitely generated free modules of $M^*$ gives
a totally acyclic complex of finitely generated free modules.  Its dual treats $M^*$ as well.
For an injective $E$, tensoring such a complex is the same as applying
$\Hom_R(-,E)$ to its exact dual.  Thus the complexes are also complete
flat resolutions.
\end{proof}

The finite factorization extends to the divisorial closure of a
finitely generated fractional submodule.

\begin{lemma}\label{lem:qcs-rank-one}
Let $R$ be a domain with quotient field $K$. If $G\subseteq K$ is a
Gorenstein flat $R$-module and $0\ne N\subseteq G$ is a finitely
generated fractional submodule, then $N_v\subseteq G$.
\end{lemma}
\begin{proof}
Embed $G$ into a flat module $F$ with Gorenstein flat, hence torsion-free,
cokernel.  Factor $N\to F$ through a finitely generated free module by
Lemma~\ref{lem:finite-factorization}(i).  Each coordinate map $N\to R$ is multiplication by an element of
$(R:N)$.  For $x\in N_v$ all these coordinate values lie in $R$, so $x$
belongs to $F$, with the identification on the line $KG\subseteq KF$.
Some $0\ne d\in R$ satisfies $dx\in N$.  Torsion-freeness of $F/G$ then
puts $x$ in $G$.
\end{proof}

\subsection{Normal localizations and quadratic integral elements}

A two-chart exact sequence relates Gorenstein flatness of overrings
to invertibility of two-generated ideals in a normal domain.

\begin{lemma}\label{lem:qcs-normal}
Let $R$ be an integrally closed domain. If every overring of $R$ is
Gorenstein flat as an $R$-module, then $R$ is Pr\"ufer.
\end{lemma}
\begin{proof}
For $0\ne a,b\in R$, put $x=b/a$, $S=R[x]$, $T=R[x^{-1}]$, and
$U=R[x,x^{-1}]$.  Integral closedness gives $S\cap T=R$, whereas
$S+T=U$; see \cite[Tag~052J]{Stacks}.  Thus
\[
 0\longrightarrow R\longrightarrow S\oplus T\longrightarrow U
 \longrightarrow0
\]
is exact.  Choose an injective cogenerator $E$ of $R$-modules; thus
$\Hom_R(N,E)\ne0$ for every nonzero $N$. Since $\Tor_1^R(E,U)=0$, the map
$E\to(E\otimes_R S)\oplus(E\otimes_R T)$ is injective.

Let $I=(a,b)R$. The trace ideal of $I$ is $II^{-1}$.
If $II^{-1}\ne R$, choose $0\ne e\in E$ killed by
$II^{-1}$.  Every relation $ra+sb=0$ is $(r,s)=(bc,-ac)$ for
$c\in I^{-1}$.  Consequently the prescriptions
$(a,b)\mapsto(e,0)$ and $(a,b)\mapsto(0,e)$ define maps $I\to E$.
Extending them to $R$ gives $e_1,e_2\in E$ with
$ae_1=e$, $be_1=0$, $ae_2=0$, $be_2=e$.
Then $e\otimes1=be_2\otimes1=ae_2\otimes x=0$ in $E\otimes_R S$,
and similarly it is zero in $E\otimes_R T$.  This contradicts the
injection.  Hence every nonzero two-generated ideal is invertible,
which is the classical Pr\"ufer criterion.
\end{proof}

The inverse of a nonprincipal divisorial maximal ideal yields a
module-finite overring with two module generators.

\begin{lemma}\label{lem:qcs-divisorial-maximal}
Let $(D,\mathfrak n,k)$ be a local domain satisfying $\OG$.  If its nonzero
maximal ideal is divisorial and not principal, then
$U=(D:\mathfrak n)$ is a proper overring, $U=D+Dx$ for some $x$, and
$(D:U)=\mathfrak n$.  In particular, $U$ is module-finite and
$\dim_k(U/\mathfrak n)=2$.
\end{lemma}
\begin{proof}
\step{Step 1: the inverse ring and algebraicity.}The ideal $\mathfrak n(D:\mathfrak n)$ equals either $\mathfrak n$ or
$D$.  The latter would make $\mathfrak n$ invertible, hence principal.
Thus $U=(\mathfrak n:\mathfrak n)$ is a ring.  Divisoriality gives
$(D:U)=\mathfrak n$ and $U\ne D$.

For $x\in U\setminus D$, maximality gives
$(D:D+Dx)=\mathfrak n$, whence $(D+Dx)_v=U$.
Lemma~\ref{lem:qcs-rank-one} applied to the Gorenstein flat overring
$D[x]$ yields $U=D[x]$.  Write $\alpha=x+\mathfrak n$.  If $\alpha$
were transcendental over $k$, then $x^2\notin D$, and the same reasoning
would give $U=D[x^2]$, contradicting $k[\alpha]\ne k[\alpha^2]$.
Thus $\alpha$ is algebraic.  Lifting a monic relation to $f(X)\in D[X]$
gives $f(x)\in\mathfrak n\subset D$; the polynomial $f(X)-f(x)$ shows
that $x$ is integral.  Hence $U$ is module-finite.  Its conductor is
finitely generated by Lemma~\ref{lem:qcs-finite-dual}.

\step{Step 2: the residue dimension.}Put $d=\dim_k(U/\mathfrak n)\ge2$, and lift a basis beginning with $1$
to $1,t_1,\ldots,t_{d-1}\in U$.  Since $\mathfrak nU=\mathfrak n\subset D$,
these elements generate $U$ over $D$, and comparison modulo
$\mathfrak n$ gives an exact sequence
\[
 0\longrightarrow\mathfrak n^{\oplus(d-1)}\longrightarrow D^d
 \longrightarrow U\longrightarrow0.
\]
As $\Ext_D^1(U,D)=0$ and $\Hom_D(\mathfrak n,D)=U$, duality gives
$D^d\twoheadrightarrow U^{\oplus(d-1)}$.  Reduction modulo
$\mathfrak n$ yields $d(d-1)\le d$, so $d=2$.  A lift of the second
basis element now gives $U=D+Dx$.
\end{proof}

\label{sec:qcs-witnesses}

A domain $D$ is \emph{seminormal} if $z^2,z^3\in D$ for
$z\in\Frac D$ implies $z\in D$; this is the standard criterion for
seminormality in a domain \cite[Tag~0EUK]{Stacks}. The following
conductor observation permits passage from a seminormal, nonnormal
localization to a divisorial maximal ideal.

\begin{lemma}\label{lem:qcs-seminormal-conductor}
Let $D$ be a seminormal domain and $U$ an overring.  Then the conductor
$(D:U)$ is a radical ideal of $U$, and hence of $D$.
\end{lemma}
\begin{proof}
Seminormality means that $z^2,z^3\in D$, $z\in\Frac D$, imply $z\in D$.
It also implies that if all sufficiently large powers of $z$ belong to
$D$, then $z\in D$.  Indeed, the least such threshold $n>1$ would put
both $(z^{n-1})^2$ and $(z^{n-1})^3$ in $D$, lowering the threshold.
For $B=(D:U)$, let $z\in U$ satisfy $z^n\in B$.  For every $u\in U$ and
$j\ge n$, $(zu)^j=z^n z^{j-n}u^j\in B\subseteq D$.  Thus $zu\in D$ for
all $u$, which means $z\in B$.
\end{proof}

The preceding inverse-ring calculation produces a quadratic witness at
every nonnormal localization, not only at minimal primes of a fixed ideal.

\begin{theorem}\label{thm:qcs-quadratic-witness}
Let $R$ be a domain satisfying $\OG$, and let $\mathfrak p\in\Spec R$.
If $R_{\mathfrak p}$ is not integrally closed, there is $z\in\Frac R$
such that
\[
 z\notin R_{\mathfrak p},\qquad z^2=u+vz\quad(u,v\in R).
\]
In particular $T=R+Rz$ is a proper module-finite overring, its conductor
$C=(R:T)$ is nonzero and finitely generated, and $C\subseteq\mathfrak p$.
\end{theorem}
\begin{proof}
\step{Step 1: a local quadratic witness.}Work over the local domain $D=R_{\mathfrak p}$, which satisfies
$\OG$.  If $D$ is not seminormal, choose $y\notin D$ with
$y^2,y^3\in D$.  Then $D+Dy$ is already a quadratic module overring.

Suppose instead that $D$ is seminormal but not normal.  Choose an
integral $y\notin D$ and put $U=D[y]$.  The conductor $B=(D:U)$ is nonzero
and finitely generated by Lemma~\ref{lem:qcs-finite-dual}; it is radical
by Lemma~\ref{lem:qcs-seminormal-conductor}.  Let $\mathfrak q$ be a
minimal prime over $B$ and set $D'=D_{\mathfrak q}$, $U'=U_{\mathfrak q}$.
Because $U$ is module-finite, conductors commute with this localization.
Hence
\[
 (D':U')=B_{\mathfrak q}=\mathfrak qD'=\mathfrak n.
\]
This is a nonzero proper ideal, so $U'\ne D'$.  It is divisorial, being
an intersection of principal fractional ideals.  It cannot be principal:
if it were, its $D'$-endomorphism ring would be $D'$, whereas it contains
$U'$ because $\mathfrak n$ is a $U'$-ideal.

Lemma~\ref{lem:qcs-divisorial-maximal} applied to $D'$ gives a proper
overring $D'+D'w$ and a relation $w^2=b+cw$ with $b,c\in D'$.
Choose $s\in D\setminus\mathfrak q$ with $sb,sc\in D$.
Then $(sw)^2=s^2b+(sc)(sw)$ has coefficients in $D$, and
$sw\notin D'$; in particular $sw\notin D$.  Thus in either case a
quadratic element outside $D$ has been found.

\step{Step 2: descent of the coefficients to $R$.}Clear the two remaining coefficient denominators from $D=R_{\mathfrak p}$:
if $w_0^2=b_0+c_0w_0$ over $D$ and $w_0\notin D$, choose
$t\in R\setminus\mathfrak p$ with $tb_0,tc_0\in R$, and put $z=tw_0$.
Then $z\notin D$ and $z^2=t^2b_0+(tc_0)z$ is a relation over $R$.
It yields $T=R+Rz$.  Its conductor is finitely generated by finite GF
duality.  If $c\in C\setminus\mathfrak p$, then $cz\in R$ would give
$z\in R_{\mathfrak p}$, a contradiction.  Thus $C\subseteq\mathfrak p$.
\end{proof}

\section{Conductor quotients and syzygy lifting}
\label{sec:qcs-return}

A proper two-generated module overring has a conductor quotient with
a normalized degree-one duality. We first obtain the IF property from
intermediate rings, then verify the biduality map of a first syzygy.
The resulting complete resolution gives a bound over $R$ for every
module annihilated elementwise by conductor powers.

\subsection{Closure properties and the IF conductor quotient}

For a commutative ring $A$, a module $E$ is \emph{FP-injective} if
$\Ext_A^1(F,E)=0$ for every finitely presented module $F$.
The ring $A$ is \emph{self-FP-injective} when this holds for $E=A$;
it is an \emph{IF-ring} when every injective $A$-module is flat.
For commutative rings, being IF is equivalent to being coherent and
self-FP-injective \cite[Theorem~2.6]{MTY}. An ideal $C$ annihilates a
module elementwise by powers when, for every element $n$, some integer
$q\ge1$ satisfies $C^qn=0$. The integer may depend on $n$.

Put $\mathcal D_1(R)=\{M:\Gfd_R M\le1\}$.
The zero module is included.  We first specify the closure facts used to
pass from finitely presented quotients to arbitrary torsion modules.

\begin{lemma}\label{lem:qcs-done}
For every commutative ring $R$, the class $\mathcal D_1(R)$ is closed under
extensions, kernels of epimorphisms between its members, direct summands,
and filtered colimits.  If $M\in\mathcal D_1(R)$ and $f\in R$, then
$M_f\in\mathcal D_1(R)$, with the dimension still taken over $R$.
Also, an ideal $I$ is Gorenstein flat if and only if
$R/I\in\mathcal D_1(R)$.
\end{lemma}
\begin{proof}
The extension, kernel, summand, and syzygy statements follow from
\cite[Theorems~2.8 and~2.11]{BennisGF} and the general closure theorem
\cite{SS}.  For filtered colimits, use the functorial
presentation by the free module on the underlying set of $M$.
Its kernel $\Omega(M)$ is Gorenstein flat precisely when
$M\in\mathcal D_1(R)$.
The underlying-set functor and the free-module functor commute with
filtered colimits. Exactness of filtered colimits then gives
$\Omega(\varinjlim M_\lambda)\cong\varinjlim\Omega(M_\lambda)$.  Apply the direct-limit
closure of Gorenstein flat modules.  Finally $M_f$ is the sequential
colimit of copies of $M$ under multiplication by $f$.  
\end{proof}

Intermediate rings determine annihilators in the conductor quotient.
The double-annihilator identity will give both coherence and
self-FP-injectivity of that quotient.

\begin{lemma}\label{lem:qcs-quadratic-if}
Let $R$ be a domain satisfying $\OG$, and let
$R\subsetneq T=R+Rx\subseteq\Frac R$ be an overring.  Put $C=(R:T)$ and
$A=R/C$.  Then $T$ and $C$ are totally reflexive $R$-modules, and $A$ is an
IF-ring.  In particular $A$ is coherent and self-FP-injective.
\end{lemma}
\begin{proof}
\step{Step 1: the dual modules.}We have $C=T^*$, and the finite duality in
Lemma~\ref{lem:qcs-finite-dual} gives $C\in\mathrm{FP}_\infty(R)$.
The exact sequence
\[
 0\longrightarrow C\xrightarrow{c\mapsto(-xc,c)}R^2
 \xrightarrow{(a,b)\mapsto a+bx}T\longrightarrow0
\]
therefore puts $T$ in $\mathrm{FP}_\infty(R)$.  It and its dual $C$ are
totally reflexive.  Also $T/R\cong A$ by $r+C\mapsto rx+R$.

\step{Step 2: intermediate rings and annihilators.}Let $J$ be a finitely generated ideal of $A$ and $I$ its inverse image in
$R$.  The module $S=R+Ix$ is a ring: write $x^2=u+vx$ with $u,v\in R$ and
multiply its elements.  Since $C$ and $J$ are finitely generated, $S$ is
module-finite, hence Gorenstein flat by $\OG$.  Put $D=(R:S)$.  Duality
gives that $D$ is finitely generated and $S=(R:D)$.  Direct calculation
shows
\[
 D/C=\operatorname{ann}_A J,\qquad
 J=\operatorname{ann}_A\operatorname{ann}_A J.
\]
For the second equality, for $r\in R$ one has
$rD\subseteq C$ if and only if $rx\in(R:D)=S$, if and only if $r\in I$.
Thus annihilators of finitely generated ideals of $A$ are finitely
generated and satisfy the double-annihilator identity.

\step{Step 3: coherence and self-FP-injectivity.}These statements imply coherence: for finitely generated $J,L$,
\[
 J\cap L=\operatorname{ann}_A
       (\operatorname{ann}_A J+\operatorname{ann}_A L)
\]
is finitely generated, and annihilators of single elements are finitely
generated.  They also give self-FP-injectivity.  To extend a map from
$J+bA$ to $A$, first extend its restriction to $J$ by induction on the
number of generators and subtract that multiplication map.  Its value
$h$ at $b$ is then annihilated by $(J:b)$.  Since
$b\operatorname{ann}_A J$ is finitely generated,
\[
 \operatorname{ann}_A(J:b)
 =\operatorname{ann}_A\operatorname{ann}_A
       (b\operatorname{ann}_A J)
 =b\operatorname{ann}_A J.
\]
Choose $d\in\operatorname{ann}_A J$ with $bd=h$; multiplication by $d$
extends the remaining map.  The base case is the double-annihilator
identity for a principal ideal.  Hence every map from a finitely
generated ideal to $A$ extends to $A$.  To pass from ideals to a finitely generated submodule $L\subseteq A^n$,
project onto the last coordinate. Its image $J$ is finitely generated,
and its kernel is finitely generated by coherence. Extend a given map
on this kernel to $A^{n-1}$ by induction, subtract that extension, and
factor the remaining map through $J$. The ideal case then extends it
to the last copy of $A$. Hence maps from all finitely generated
submodules of finite free modules extend. Thus $A$ is self-FP-injective.  The equivalence with the IF property is
\cite[Theorem~2.6]{MTY}.
\end{proof}

\subsection{Degree-one duality and the first syzygy}

The following comparison fixes the extension class and the normalization
of the connecting maps used in the syzygy calculation.

\begin{lemma}\label{lem:quotient-degree-one}
Let $R$ be a domain and let $T=R+Rx\subseteq\Frac R$ be a proper
overring that is Gorenstein flat over $R$. Put $C=(R:T)$ and $A=R/C$.
For every $A$-module $L$ there are natural isomorphisms
\begin{equation}\label{eq:qcs-ext-shift}
 \Ext_R^{i+1}(L,R)\cong\Ext_A^i(L,A)\quad(i\ge0),
 \qquad \Hom_R(L,R)=0.
\end{equation}
The degree-one isomorphism sends $\phi:L\to A$ to the pullback of
$0\to R\to T\to A\to0$, where $T\to A$ sends $r+sx$ to $s+C$.
\end{lemma}
\begin{proof}
Finite duality and the two-generator presentation give that $T,C$
are totally reflexive over $R$, and
$A\in\mathrm{FP}_\infty(R)$.  The sequence $0\to C\to R\to A\to0$
gives
\[
 \Hom_R(A,R)=0,\qquad \Ext_R^1(A,R)\cong T/R\cong A,\qquad
 \Ext_R^i(A,R)=0\quad(i>1).
\]
Choose the last identification so that the class of
\[
 \eta:\quad0\longrightarrow R\longrightarrow T
 \longrightarrow A\longrightarrow0
\]
corresponds to $1$.  This is possible because the connecting map
$R\to\Ext_R^1(A,R)$ for $\eta$ is surjective with kernel $C$.
Indeed, $T^*=C$ and $\Ext_R^1(T,R)=0$.

The vanishing of $\Hom_R(L,R)$ follows because $C\ne0$ annihilates
$L$ and $R$ is torsion-free. Take an injective resolution $E^\bullet$
of $R$. The complex $\Hom_R(A,E^\bullet)$ consists of injective
$A$-modules and has its only cohomology in degree one, equal to $A$.
Its degree-zero monomorphism splits, because its source is injective.
Delete the resulting contractible summand and shift: the remaining
complex is an injective resolution of $A$.  Applying $\Hom_A(L,-)$
proves \eqref{eq:qcs-ext-shift}.  In degree one the identification sends
$\phi:L\to A$ to the pullback $\phi^*\eta$.

\end{proof}

We now verify the natural evaluation map, in addition to Ext vanishing.
This is the step that permits a complete resolution over $R$.

\begin{lemma}\label{lem:qcs-syzygy-lift}
Let $R$ be a domain and let $R\subsetneq T=R+Rx\subseteq\Frac R$
be an overring that is Gorenstein flat over $R$. Set $C=(R:T)$ and
$A=R/C$, and suppose that $A$ is an IF-ring.  If $N$ is a finitely presented $A$-module and
\[
 0\longrightarrow H\longrightarrow F\xrightarrow{\pi}N
 \longrightarrow0
\]
is a presentation with $F$ finitely generated free over $R$, then $H$ is
totally reflexive over $R$.
\end{lemma}
\begin{proof}
\step{Step 1: finite resolutions and the dual presentation.}
Lemma~\ref{lem:quotient-degree-one} applies, and its proof shows
$A\in\mathrm{FP}_\infty(R)$.
Since $A$ is coherent and self-FP-injective, every finitely presented
$A$-module is totally reflexive; use \cite[Theorem~2.6]{MTY} and finite
free resolutions over the coherent ring.  Put $N^\vee=\Hom_A(N,A)$.
Thus $N,N^\vee\in\mathrm{FP}_\infty(A)$,
$N\cong N^{\vee\vee}$, and their positive Ext groups into $A$ vanish.
They also belong to $\mathrm{FP}_\infty(R)$.  To justify this last
assertion, all finite free $A$-modules belong to
$\mathrm{FP}_\infty(R)$; induction on $n$ along an $A$-resolution
by finitely generated free modules shows that every $\mathrm{FP}_\infty(A)$-module is
$\mathrm{FP}_n(R)$, for every $n$.

It follows that $H\in\mathrm{FP}_\infty(R)$ and
$\Ext_R^{>0}(H,R)=0$.  Dualizing the presentation of $N$ gives
\begin{equation}\label{eq:qcs-first-dual}
 0\longrightarrow F^*\longrightarrow H^*\longrightarrow N^\vee
 \longrightarrow0.
\end{equation}
In particular $H^*\in\mathrm{FP}_\infty(R)$.  Dualizing again and using
\eqref{eq:qcs-ext-shift} gives
\begin{equation}\label{eq:qcs-second-dual}
 0\longrightarrow H^{**}\longrightarrow F
 \xrightarrow{\beta}N^{\vee\vee}
 \longrightarrow\Ext_R^1(H^*,R)\longrightarrow0.
\end{equation}
\step{Step 2: identification of the biduality map.}We verify that $\beta$ is the composite of $\pi$ and evaluation
$N\to N^{\vee\vee}$, after making compatible choices of signs for the connecting maps.

For $\phi\in N^\vee$, lift $\phi\pi:F\to A$ through $T\to A$.
The restriction of a lift $\widetilde\phi:F\to T$ to $H$ lands in $R$
and represents $\phi$ in the quotient in
\eqref{eq:qcs-first-dual}.  Conversely, for each $h\in H^*$ choose
$\phi$ corresponding to its class.  Adjusting such a lift by a map
$F\to R$ makes its restriction equal to $h$.  This extension of $h$ to
$F\to T$ is unique, since $\Hom_R(N,T)=0$: $N$ is killed by the nonzero
ideal $C$, whereas $T$ is torsion-free.  For $u\in F$, evaluation of
these extensions at $u$ gives a map $H^*\to T$ extending
$F^*\to R$, $q\mapsto q(u)$.  Pushing out
\eqref{eq:qcs-first-dual} by this latter map identifies its extension
class with the pullback of $\eta$ along
\[
 N^\vee\longrightarrow A,\qquad
 \phi\longmapsto\phi(\pi(u)).
\]
This is precisely the connecting map defining $\beta(u)$.  Hence
$\beta(u)(\phi)=\phi(\pi(u))$, as claimed.

\step{Step 3: the complete resolution.}Since evaluation $N\to N^{\vee\vee}$ is an isomorphism,
\eqref{eq:qcs-second-dual} identifies $H^{**}$ with $H$ through its
natural map and gives $\Ext_R^1(H^*,R)=0$.  For $i>1$,
\eqref{eq:qcs-first-dual} and \eqref{eq:qcs-ext-shift} give
$\Ext_R^i(H^*,R)=0$.  Thus $H$ and $H^*$ are
$\mathrm{FP}_\infty$, $H\cong H^{**}$, and their positive Ext groups
into $R$ vanish.  Splice a resolution by finitely generated free modules of $H$ with the
dual of one of $H^*$ to obtain the required totally acyclic complex.
Applying $\Hom_R(-,E)$ to its exact dual for an injective $E$
shows that the complex is a complete flat resolution as well.
\end{proof}

\subsection{Modules annihilated by conductor powers}

The syzygy construction initially concerns finitely presented quotient
modules. Filtered-colimit and extension closure extend its conclusion
to all quotient modules and then to modules annihilated elementwise by
conductor powers, always over the original ring $R$.

\begin{theorem}\label{thm:qcs-return}
Let $R$ be a domain, and let $R\subsetneq T=R+Rx\subseteq\Frac R$
be an overring that is Gorenstein flat as an $R$-module,
put $C=(R:T)$, and suppose that $R/C$ is an IF-ring.  Then
\[
 \Gfd_R N\le1\qquad\text{for every }R/C\text{-module }N.
\]
More generally the same bound holds for every $R$-module $N$ such that
each element of $N$ is annihilated by some power of $C$.
Consequently, if $R$ satisfies $\OG$, these assertions apply to every
proper two-generated module overring.  For every ideal $J$ of $R$ they give
\[
 C\subseteq\sqrt J\quad\Longrightarrow\quad J\in\GF(R).
\]
\end{theorem}
\begin{proof}
The finitely presented $R/C$-modules have the asserted bound by
Lemma~\ref{lem:qcs-syzygy-lift}.  Every $R/C$-module is a filtered colimit
of finitely presented modules, so Lemma~\ref{lem:qcs-done} gives the bound
for all of them, over $R$.

If $C^nN=0$, the finite filtration by $C^jN$ has factors annihilated by
$C$; extension closure gives $N\in\mathcal D_1(R)$.  In the elementwise
case take the increasing union of $\{y\in N:C^ny=0\}$ and use filtered
colimits.  Under $\OG$, Lemma~\ref{lem:qcs-quadratic-if} supplies the IF
hypothesis.  Finally $C$ is finitely generated, so $C\subseteq\sqrt J$
implies $C^n\subseteq J$ for some $n$.  Apply the bound to $R/J$ and then
Lemma~\ref{lem:qcs-done}.
\end{proof}

\section{Finite patching and the overring characterization}
\label{sec:cp-patching}

We now combine the open neighborhoods supplied by normal localizations
with the closed conductor sets supplied by nonnormal localizations.
Constructible compactness makes the data finite; a bounded localization
complex then gives the required dimension bound. All module dimensions
in this section are computed over the fixed base ring $R$.

\subsection{A finite localization complex}

Products of conductor ideals preserve the torsion bound that will be
needed for the cohomology of a localization complex.

\begin{lemma}\label{lem:cp-product-torsion}
Let $R$ be a commutative ring and let $C_1,\ldots,C_s$ be ideals such that
every $R/C_j$-module belongs to $\mathcal D_1(R)$.  Put
$P=C_1\cdots C_s$.  Then every $R$-module whose elements are annihilated
by powers of $P$ belongs to $\mathcal D_1(R)$.
For $s=0$, use $P=R$; the assertion then concerns only the zero module.
\end{lemma}
\begin{proof}
Suppose first that $P^nN=0$ for a fixed $n$.  Starting with $N$, multiply
successively by $C_1,\ldots,C_s$, and repeat this list $n$ times.  This
gives a finite filtration ending in $P^nN=0$.  Each factor is annihilated
by one $C_j$, so extension closure puts $N$ in $\mathcal D_1(R)$.
For elementwise annihilation, write
$N=\bigcup_{n\ge1}(0:_N P^n)$ and use filtered-colimit closure from
Lemma~\ref{lem:qcs-done}.  The empty-product convention is immediate.
\end{proof}

\Needspace{8\baselineskip}
A descending kernel argument in a bounded complex preserves the
Gorenstein flat dimension bound.

\begin{lemma}\label{lem:cp-bounded-complex}
Let $R$ be a commutative ring and let
$X^0\to X^1\to\cdots\to X^r$ be a cochain complex with zero terms outside
$0,\ldots,r$.  Suppose that $X^i\in\mathcal D_1(R)$ for $i>0$, and that
$H^i(X^\bullet)\in\mathcal D_1(R)$ for every $i\ge0$.  Then
$X^0\in\mathcal D_1(R)$.
\end{lemma}
\begin{proof}
The case $r=0$ is immediate.  Write
$Z^i=\ker d^i$ and $B^i=\operatorname{im}d^{i-1}$.
At the top, $Z^r=X^r$ and
$B^r=\ker(X^r\to H^r(X^\bullet))$ belong to $\mathcal D_1(R)$.
Descending induction uses
\[
 0\longrightarrow Z^i\longrightarrow X^i\longrightarrow B^{i+1}
 \longrightarrow0,\qquad
 0\longrightarrow B^i\longrightarrow Z^i\longrightarrow H^i(X^\bullet)
 \longrightarrow0
\]
for $i=r-1,\ldots,1$.  Kernel closure gives $B^1\in\mathcal D_1(R)$.
Finally $0\to H^0(X^\bullet)\to X^0\to B^1\to0$ and extension closure
prove the result.  All these closure properties are those of
Lemma~\ref{lem:qcs-done}.
\end{proof}

The following patching statement allows closed conductor pieces as well
as ordinary basic-open pieces.  Its hypotheses are bounds over $R$, not
bounds computed over different local rings.

\begin{proposition}\label{prop:cp-finite-patching}
Let $R$ be a commutative ring, let $I$ be an ideal, and put $B=R/I$.
Suppose that $f_1,\ldots,f_r\in R$ and finitely generated ideals
$C_1,\ldots,C_s$ satisfy the following conditions:
\begin{enumerate}
\item $B_{f_i}\in\mathcal D_1(R)$ for every $i$;
\item every $R/C_j$-module belongs to $\mathcal D_1(R)$ for every $j$;
\item the prime ideals containing $I$ are covered as follows:
\[
 V(I)\subseteq\bigcup_{i=1}^r D(f_i)\ \cup\
                 \bigcup_{j=1}^s V(C_j).
\]
\end{enumerate}
Then $B\in\mathcal D_1(R)$.  Either indexing set is allowed to be empty.
\end{proposition}
\begin{proof}
\step{Step 1: the finite annihilator bound.}Put $J=(f_1,\ldots,f_r)$ and $P=C_1\cdots C_s$, with $P=R$ if $s=0$.
The covering gives
\[
 V(I+J)\subseteq V(P),\qquad P\subseteq\sqrt{I+J}.
\]
The ideal $P$ is finitely generated, so some $e\ge1$ satisfies
\begin{equation}\label{eq:cp-power-containment}
 P^e\subseteq I+J.
\end{equation}
Every module annihilated by $I$ and elementwise annihilated by powers
of $J$ is therefore elementwise annihilated by powers of $P$.
It belongs to $\mathcal D_1(R)$ by Lemma~\ref{lem:cp-product-torsion}.

\step{Step 2: the localization complex and its cohomology.}If $r=0$, \eqref{eq:cp-power-containment} gives $P^eB=0$, proving the
result.  Suppose $r\ge1$.  Form the finite augmented localization complex
\begin{equation}\label{eq:cp-cech}
 X^\bullet
 = B\otimes_R\bigotimes_{i=1}^r[\,R\longrightarrow R_{f_i}\,],
\end{equation}
where each two-term factor has degrees zero and one, with the usual
signs on the tensor product differential.  Thus
\[
 X^0=B,\qquad
 X^q=\bigoplus_{1\le i_1<\cdots<i_q\le r}
        B_{f_{i_1}\cdots f_{i_q}}\quad(1\le q\le r).
\]
For $q>0$, every summand is a localization of one of the modules
$B_{f_i}$ in the hypothesis, so $X^q\in\mathcal D_1(R)$ by
Lemma~\ref{lem:qcs-done}.

After localizing \eqref{eq:cp-cech} at $f_i$, its $i$th two-term factor
is the identity complex $[R_{f_i}\xrightarrow{1}R_{f_i}]$.
Consequently $X^\bullet_{f_i}$ is contractible, and
$H^q(X^\bullet)_{f_i}=0$ for every $i,q$.  For any cohomology element
$h$, choose $n_i\ge1$ with $f_i^{n_i}h=0$.
Then $J^Nh=0$ for $N=1+\sum_i(n_i-1)$, by expansion of products of the
generators of $J$.  Also $Ih=0$, since the complex consists of
$B$-modules.  Equation~\eqref{eq:cp-power-containment} gives
$P^{eN}h=0$.  Hence all $H^q(X^\bullet)$ belong to $\mathcal D_1(R)$.
\step{Step 3: descending through the bounded complex.}
Lemma~\ref{lem:cp-bounded-complex} gives $X^0=B\in\mathcal D_1(R)$.
This uses only extensions and kernels of epimorphisms between modules
in $\mathcal D_1(R)$, not unrestricted closure under cokernels.
The argument does not identify the displayed cohomology with derived
local cohomology and imposes no weak-proregularity condition.
\end{proof}

\subsection{Constructible compactness and the global characterization}

The required finite covering is obtained in the constructible topology.
Recall that this topology on $\Spec R$ is generated by the sets $D(f)$
and $V(f)$, both of which are open and closed in it.  The spectrum with
this topology is compact; see \cite[Tag~0901]{Stacks}.
A closed subset $V(I)$ is also compact in this topology, even when $I$
is not finitely generated.  Finiteness of the ideals $C_j$ makes $V(C_j)$
constructibly open, which is the essential distinction from an arbitrary
Zariski-closed covering.

\begin{lemma}\label{lem:cp-cover-selection}
Let $R$ be a commutative ring and $I$ an ideal.  Suppose that every
$\mathfrak p\in V(I)$ belongs to a member of a family of subsets of
$\Spec R$, each of the form $D(f)$ or $V(C)$ with $C$ finitely generated.
Then finitely many members of this family cover $V(I)$.
\end{lemma}
\begin{proof}
Each indicated set is open in the constructible topology.
The subset $V(I)=\bigcap_{a\in I}V(a)$ is closed in that compact
Hausdorff space.  Apply compactness of $V(I)$.
\end{proof}

\label{sec:cp-characterization}

We apply the finite patching criterion to every finitely generated ideal.
At a normal localization its quotient has a basic-open flat resolution.
At a nonnormal localization a quadratic conductor supplies a closed
piece.  These pieces need not be open in the Zariski topology.

\begin{proposition}\label{prop:cp-finite-certificate}
Let $R$ be a domain satisfying $\OG$, and let $I$ be a nonzero finitely
generated proper ideal.  There are finitely many elements
$f_1,\ldots,f_r$ and finitely many proper overrings
$T_j=R+Rz_j$ with conductors $C_j=(R:T_j)$ such that
\begin{enumerate}
\item $I_{f_i}=a_iR_{f_i}$ for some $0\ne a_i\in I$, for every $i$;
\item each $C_j$ is nonzero and finitely generated, and every
$R/C_j$-module has Gorenstein flat dimension at most one over $R$;
\item $V(I)\subseteq\bigcup_iD(f_i)\cup\bigcup_jV(C_j)$.
\end{enumerate}
In particular, for some $e\ge1$,
\[
 (C_1\cdots C_s)^e\subseteq I+(f_1,\ldots,f_r),
\]
using the same empty-product convention as above.  
\end{proposition}
\begin{proof}
Fix $\mathfrak p\supseteq I$.  If $R_{\mathfrak p}$ is integrally closed,
then it is a valuation domain by Lemma~\ref{lem:qcs-normal}, since
$\OG$ localizes.  Some nonzero generator $a\in I$ generates
$I_{\mathfrak p}$.  Clearing the coefficient denominators of the finitely
many generators of $I$ yields $f\notin\mathfrak p$ with $I_f=aR_f$.
This supplies a set $D(f)$ containing $\mathfrak p$.

If $R_{\mathfrak p}$ is not integrally closed,
Theorem~\ref{thm:qcs-quadratic-witness} supplies a proper quadratic module
overring $T=R+Rz$ with a nonzero finitely generated conductor
$C\subseteq\mathfrak p$.  Lemma~\ref{lem:qcs-quadratic-if} and
Theorem~\ref{thm:qcs-return} give the asserted bound for all $R/C$-modules.
This supplies the set $V(C)$ containing $\mathfrak p$.

Doing this for every prime over $I$ gives a constructibly open covering.
Lemma~\ref{lem:cp-cover-selection} extracts a finite subcover.
The final containment follows as in
\eqref{eq:cp-power-containment}.  
\end{proof}

\Needspace{6\baselineskip}
The criterion now applies globally, with all dimensions computed over $R$.

\begin{theorem}\label{thm:cp-main}
Let $R$ be a commutative integral domain with quotient field $K$.
The following conditions are equivalent:
\begin{enumerate}
\item $R$ is a Gorenstein Pr\"ufer domain;
\item every overring $R\subseteq T\subseteq K$ is Gorenstein flat as an
$R$-module;
\item $R[z_1,\ldots,z_n]$ is Gorenstein flat as an $R$-module for every
finite set $z_1,\ldots,z_n\in K$.
\end{enumerate}
In particular, the overring hypothesis forces coherence and
$\operatorname{fd}_R E\le1$ for every injective $R$-module $E$.
\end{theorem}
\begin{proof}
If $R$ is $G$-Pr\"ufer, every torsion-free $R$-module is Gorenstein flat
by \cite[Corollary~2.4]{HLZ}; this includes every overring.  Thus
(i) implies (ii), and (ii) implies (iii).  Every overring is the directed
union of its finitely generated $R$-subalgebras, so the direct-limit
closure in \cite{SS} proves (iii) implies
(ii).

Assume (ii).  It suffices to prove that every finitely generated ideal
of $R$ is Gorenstein flat.  The zero and unit ideals are immediate.
For a nonzero finitely generated proper ideal $I$, use the data in
Proposition~\ref{prop:cp-finite-certificate}.  Each $I_{f_i}=a_iR_{f_i}$
gives an exact sequence of $R$-modules
\[
 0\longrightarrow R_{f_i}\xrightarrow{a_i}R_{f_i}
 \longrightarrow(R/I)_{f_i}\longrightarrow0.
\]
Both copies of $R_{f_i}$ are flat over $R$, and the first map is injective
because $R$ is a domain and $a_i\ne0$.  Hence
$(R/I)_{f_i}\in\mathcal D_1(R)$.  The other hypotheses of
Proposition~\ref{prop:cp-finite-patching} are supplied by the same
finite certificate.  That proposition yields
$\Gfd_R(R/I)\le1$.  The syzygy criterion therefore makes $I$ Gorenstein
flat.

The ideal characterization in Section~\ref{sec:qcs-preliminaries}
now gives the $G$-Pr\"ufer property, including coherence and
$\fd_R E\le1$ for every injective $E$. This proves (ii) implies (i).
\end{proof}

The characterization recovers the classical conclusion under normality
and also gives complete resolutions for the finitely generated ideals.

\begin{corollary}\label{cor:cp-normal-conclusion}
An integrally closed domain whose overrings are Gorenstein flat is a
classical Pr\"ufer domain.  For an arbitrary domain satisfying the overring
condition, every finitely generated ideal has a complete resolution by
finitely generated projective modules, and every torsion-free module is
Gorenstein flat.
\end{corollary}
\begin{proof}
The normal assertion is Lemma~\ref{lem:qcs-normal}.
In general Theorem~\ref{thm:cp-main} gives coherence and Gorenstein
flatness of every finitely generated ideal.  Such ideals belong to
$\mathrm{FP}_\infty(R)$, so Lemma~\ref{lem:qcs-finite-dual} gives their
complete resolutions by finitely generated free modules.  The torsion-free assertion is
\cite[Corollary~2.4]{HLZ}.
\end{proof}

\subsection{Separating examples and concluding remarks}

The following examples distinguish the conclusion from classical
Pr\"ufer theory and show why algebra-finiteness cannot be replaced by
module-finiteness in the testing family.

A singular curve already separates Gorenstein flatness from flatness.

\begin{example}\label{ex:overring-cusp}
Let $k$ be a field, $R=k[[t^2,t^3]]$, and $V=k[[t]]$.
Then every overring of $R$ is Gorenstein flat over $R$, but $R$ is not
Pr\"ufer and the module-finite normalization $V=R+Rt$ is not flat over $R$.
\end{example}
\begin{proof}
The presentation $R\cong k[[u,v]]/(v^2-u^3)$ makes $R$ a
one-dimensional Noetherian Gorenstein local domain. Its torsion-free
modules are Gorenstein flat by the standard Gorenstein and domain
characterizations \cite{CFH,HLZ}. The element $t=t^3/t^2$ is integral
over $R$ but does not belong to $R$, so $R$ is not integrally closed
and hence is not Pr\"ufer. If $V$ were flat, its finite presentation
over the Noetherian local ring $R$ would make it free. Its rank is one,
whereas $V/(t^2,t^3)V=V/t^2V$ has $k$-dimension two. This is impossible.
\end{proof}

Even a regular normal domain may pass every module-finite overring
test without satisfying the theorem's hypotheses.

\begin{example}\label{ex:modulefinite-insufficient}
Let $k$ be a field and let $R=k[u,v]_{(u,v)}$.
Its only module-finite overring is $R$, but the finite-type overring
$T=R[v/u]$ satisfies $\fd_RT=\Gfd_RT=1$.
In particular, testing only module-finite overrings does not
characterize Gorenstein Pr\"ufer domains.
\end{example}
\begin{proof}
The ring $R$ is an integrally closed domain, so a module-finite overring,
being integral, must equal $R$. Gauss's lemma in the UFD $R$ gives
$T\cong R[Z]/(uZ-v)$ and an exact sequence of free $R$-modules
\[
 0\longrightarrow R[Z]\xrightarrow{\,uZ-v\,}R[Z]
 \longrightarrow T\longrightarrow0.
\]
Tensoring with the residue field $k$ makes the displayed differential
zero, so $\Tor_1^R(k,T)\cong k[Z]\ne0$ and $\fd_RT=1$.
A module of finite flat dimension has the same Gorenstein flat dimension
\cite[Theorem~2.2]{Bennis}. Thus $T$ is not Gorenstein flat over $R$.
It cannot be module-finite, since it is a proper overring of $R$.
\end{proof}

The finite-patching proposition separates the two inputs to the main
theorem: local ring theory supplies principal neighborhoods and finite
conductors, while the module category supplies a resolving class closed
under filtered colimits. The localization complex never changes the
ring over which the dimension is measured. Consequently the theorem
does not depend on a descent assertion for arbitrary Gorenstein flat
modules, on Noetherian induction, or on periodic resolutions. The
examples show that the resulting characterization retains genuinely
Gorenstein behavior and that module-finite tests alone do not suffice.

\bigskip

\paragraph*{\bf Competing interests:} The author declares that there are no competing interests.

\medskip

\paragraph*{\bf Availability of data and material:}
No datasets were generated or analysed during the current study.

\medskip

\paragraph*{\bf Declaration of generative AI and AI-assisted technologies in the manuscript preparation process:}  
During the preparation of this manuscript, the author used ChatGPT (OpenAI, GPT-6) for language editing, LaTeX formatting, literature exploration, and assistance in examining mathematical arguments and proofs. All AI-assisted content was independently verified by the author, who takes full responsibility for the accuracy, originality, and integrity of the manuscript.

\end{document}